\documentclass[10pt]{article}
\usepackage{amsmath}
\usepackage{amssymb}
\usepackage{indentfirst}
\usepackage[latin1]{inputenc}
\usepackage{geometry}
\usepackage{amsfonts}
\usepackage{graphicx}
\usepackage{subfigure}
\usepackage{float}
\usepackage{color}
\usepackage{verbatim} 
\usepackage{hyperref}
\hypersetup{colorlinks,linktocpage}
\usepackage{cancel}

\newcommand{\Z}{{\mathbb Z}}

\newcommand{\Q}{{\mathbb Q}}
\newcommand{\R}{{\mathbb R}}
\newcommand{\C}{{\mathbb C}}

\newcommand{\F}{{\mathcal F}}

\newcommand{\HQ}{{\mathbb H}}

\newcommand{\h}{{\mathcal{H}}}   

\newcommand{\bc}{\begin{center}}
\newcommand{\ec}{\end{center}}

\newcommand{\Fe}{{\cal F}}

\newcommand{\SL}{{\rm SL}}
\newcommand{\PSL}{{\rm PSL}}

\newcommand{\PSU}{{\rm PSU}}

\newcommand{\Iso}{{\rm ISO}}
\newcommand{\tr}{{\rm tr}}

\newcommand{\ISO}{{\rm ISO}}

\renewcommand{\O}{\mathcal{O}}

\newcommand{\qed}{\enspace\vrule  height6pt  width4pt  depth2pt}
\newenvironment{proof}{\par\noindent{\bf Proof.}}{$\qed$\par\bigskip}

\newtheorem{theorem}{Theorem}[section]
\newtheorem{definition}[theorem]{Definition}
\newtheorem{lemma}[theorem]{Lemma}
\newtheorem{corollary}[theorem]{Corollary}
\newtheorem{proposition}[theorem]{Proposition}
\newtheorem{remark}[theorem]{Remark}

\begin{document}

\title{Dirichlet-Ford domains and Double Dirichlet domains\thanks{Mathematics subject Classification Primary
[$20H10$,$30F40$]; Secondary [$51M10$, $57M60$].
Keywords and phrases: Hyperbolic space, Kleinian groups, Fundamental domains.
\newline The first author is supported  by Onderzoeksraad of
Vrije Universiteit Brussel and Fonds voor 
Wetenschappelijk Onderzoek (Flanders),  the second by  FAPESP-Brazil (while visiting the Vrije Universiteit Brussel),  the third by Fonds voor
Wetenschappelijk Onderzoek (Flanders)-Belgium (while visiting Universit\"at Bielefeld) and   the fourth by FAPESP and CNPq-Brazil, and the fifth  by FAPESP (Funda\c c\~ao  de Amparo \`a
Pesquisa do Estado de S\~ao Paulo), Proc. 2014/06325-1.
 }}

\author{E. Jespers \and S. O. Juriaans \and  A. Kiefer \and A. De A. E Silva \and A. C. Souza Filho}
\date{}

\maketitle

\begin{abstract}
We continue investigations started by Lakeland on Fuchsian and Kleinian groups which have  a Dirichlet fundamental domain that also is a Ford domain in the upper half-space model of hyperbolic $2$- and $3$-space, or which have a Dirichlet domain with multiple centers. Such domains are called DF-domains and Double Dirichlet domains respectively. Making use of  earlier obtained concrete formulas for the bisectors defining the Dirichlet domain  of center $i \in \HQ^2$ or center $j \in \HQ^3$, we obtain a simple condition on the matrix entries of the side-pairing transformations of the fundamental domain of a Fuchsian or Kleinian group to be a DF-domain. Using the same methods, we also complement  a result of Lakeland stating that a cofinite Fuchsian group has a DF domain (or a Dirichlet domain with multiple centers) if and only if it is an index $2$ subgroup of the discrete group G of reflections in a hyperbolic polygon.
 \end{abstract}

\section{Introduction}

Fundamental domains in hyperbolic spaces, or spaces of constant curvature in general, have been studied for a long time. Together with Poincar\'e's Polyhedron Theorem, they are often used to find presentations of discrete groups. Fundamental domains also are of use in the construction of discrete groups. Special attention goes to fundamental domains in hyperbolic $2$- and $3$-space, as they are strongly related to discrete subgroups of $\PSL_2(\R)$ and $\PSL_2(\C)$. 
 
A major difficulty one encounters is the effective construction of such a  domain.
The two most known and used constructions in hyperbolic space are the Ford and Dirichlet fundamental domains. The Ford fundamental domain of a group $\Gamma$ is defined in terms of the isometric spheres of the elements of $\Gamma$. The Dirichlet domain of $\Gamma$ is based on the bisectors of some chosen center and its images by $\Gamma$. It is well-known that in the ball model of hyperbolic $2$- or $3$-space, the Dirichlet fundamental domain and the Ford fundamental domain of some discrete group $\Gamma$ are the same (see for example \cite[Theorem 9.5.2]{beardon}). In the upper half-space model however, this is, in general, not the case. Hence an interesting topic is to study when the Dirichlet and the Ford fundamental domain coincide also in the upper half-space model. This is what we call a DF domain, i.e. a fundamental domain in $\HQ^{2}$ or $\HQ^{3}$ that is a Dirichlet and a Ford domain at the same time.

One major problem in the construction of a Dirichlet fundamental domain is the choice of an adequate center. In general changing the center, changes the shape of the fundamental domain completely, as is nicely shown by Martin Deraux's animation \cite{Martin}. This is different for the hyperbolic reflection groups. Their fundamental domain is canonical and the choice of the center plays no role at all. The fundamental domain is the same for every chosen center, see \cite[Exercise 7.1.1]{ratcliffe}. So one may wonder which other Dirichlet fundamental domains have multiple centers. We call them double Dirichlet domains. In some sense characterizing a group acting discontinuously on hyperbolic space by having a double Dirichlet domain comes down to study `how close' the group is to a hyperbolic reflection group.

One of the most common examples of a discontinuous action on hyperbolic space is the action of $\PSL_2(\Z)$ on hyperbolic $2$-space $\HQ^2$. The probably best known fundamental domain for this action is the triangle in the upper half-plane model with vertices $\infty$, $\frac{1}{2}+\frac{\sqrt{3}}{2}i$ and $-\frac{1}{2}+\frac{\sqrt{3}}{2}i$. This is in fact both a Ford domain and a Dirichlet domain with center $ti$ for every $t>0$. So this is a first example of a DF-domain and a double Dirichlet domain.

In this paper we continue the investigations initiated by Lakeland in \cite{lakeland} on DF-domains and double Dirichlet domains. The main reason in \cite{lakeland} to study these domains is answering a  question raised by Agol-Belolipetsky-Storm-Whyte in \cite{agol}: the existence of  a maximal arithmetic hyperbolic reflection group which is not congruence. The author constructs a non-congruence arithmetic group $\Gamma_{\textrm{ref}}$ and, by using the theory of DF-domains, he proves that $\Gamma_{\textrm{ref}}$ is a maximal reflection group.

The main theorem of \cite[Theorem 5.3]{lakeland} states that a finitely generated, finite coarea Fuchsian group $\Gamma$ admits a DF-domain  
if and only if $\Gamma$ is an index 2 subgroup of a  reflection group. 
It  also is proved that a Kleinian group $\Gamma$ has a generating set consisting of elements whose traces are real (\cite[Theorem 6.3]{lakeland}.) 
We give a new and independent criterion for the result of \cite{lakeland} that also applies to Kleinian groups. Note that all the groups we are working with are non-cocompact. See also Remark~\ref{non-cocompact}.
Our criterion is of algebraic nature and  easily can be  checked once a set of side-pairing transformations is given:

\begin{theorem}
\label{lakeland} 
Let $\Gamma$ be a non-cocompact cofinite discrete subgroup of $\PSL_2(\C )$, acting on $\HQ^2$, respectively $\HQ^3$, and $P_0=i$, respectively $P_0=j$. Suppose that the stabilizer of $P_0$ in $\Gamma$ is trivial.  Then, $\Gamma$ admits a DF domain $\Fe$ with center $P_0$ if and only if for every side-pairing transformation $\gamma =\begin{pmatrix}
a & b\\
c & d
\end{pmatrix}\in \PSL_2(\C )$ of $\Fe$ we have that $d={\overline{a}}$.
Moreover, if $\Gamma$ is a cofinite Fuchsian group, then  ${\tilde{\Gamma}} =\langle \sigma , \Gamma \rangle$ is a reflection group and ${\hat{\Gamma}} =\langle \tau, \Gamma \rangle$ is 
a Coxeter group, where $\sigma$ is the reflection in the imaginary axis and $\tau$ is the linear operator represented by the matrix $\begin{pmatrix} i &0 \\ 0 & -i \end{pmatrix}$, and both groups  contain $\Gamma $ as a subgroup of index two.
\end{theorem}

For such groups $\Gamma$, as an immediate consequence, one obtains, in the Kleinian case, that the traces of the generating elements  are real (see 
Corollary~\ref{TraceRealC}).
Moreover, as an application, we get most of the results on DF and double Dirichlet domains obtained in \cite{lakeland}. Also, some of the proofs of \cite{lakeland} can be simplified using Theorem~\ref{lakeland}.

To prove the above theorem, we make use of concrete formulas given in \cite{algebrapaper}. In \cite{algebrapaper},  the authors develop explicit formulas for the bisectors of $i$ and $\gamma(i)$ for some $\gamma \in \PSL_2(\R)$ not fixing $i$, or $j$ and $\gamma(j)$ for some $\gamma \in \PSL_2(\C)$ not fixing $j$. These bisectors are indeed necessary to determine the Dirichlet fundamental domain in hyperbolic $2$-space $\HQ^2$ and hyperbolic $3$-space $\HQ^3$ respectively.
These fundamental domains are then used to tackle the non-trivial problem of describing units in an order of a non-commutative non-split division algebra or of a $2$-by-$2$ matrix ring over a quadratic imaginary extension of the field $\Q$.
As the  unit groups of some of these orders may be considered as discrete subgroups of $\SL_2(\C)$, fundamental domains and Poincar\'e's Polyhedron Theorem are of potential use to determine these unit groups.
First attempts to this were done by Pita, del R\'io and Ruiz in \cite{PitadelRio2006,PitadelRioRuiz2005}, where the authors use Ford domains to get presentations of some small subgroups of congruence subgroups of Bianchi groups and by Corrales, Jespers, Leal and del R\'io in \cite{jesetall}, where a presentation for the unit group of a ``small'' non-commutative division algebra (a quaternion algebra) is given using Dirichlet domains. 
As an application, one obtains a description of subgroups of finite index in the unit group of an integral group ring of some finite
groups. 
Making use of concrete formulas, the authors obtain in \cite{algebrapaper}
a more general approach than  the isolated cases described above.

The outline of the  paper is as follows. For the sake of completeness, we record in Section~\ref{secback} some fundamentals on hyperbolic geometry, fundamental domains and on  Fuchsian and Kleinian groups. In Section~\ref{seciso}, we recall a result from \cite{algebrapaper} and develop a necessary proposition to prove our main result and its corollaries. In particular, we give conditions for an isometric sphere, in the upper half-space (-plane) model, to be a bisector.
In the last section, we consider  DF domains and double Dirichlet domains, prove  Theorem~\ref{lakeland} and show some corollaries.


\section{Background}
\label{secback}

We begin by recalling basic facts on hyperbolic spaces and we fix
notation. Standard references are \cite{beardon, bridson, 
elstrodt, gromov, mac, ratcliffe}. By  $\HQ^2$ and $\HQ^3$ we denote the upper half-space model of hyperbolic $2$- and $3$-space.
As is common, we identify $\HQ^2$ with the subset  $\{x+ri
\in \C \mid x \in \R, r \in \R^{+}\}$ of the complex numbers $\C$ and  $\HQ^{3}$ with the subset $  \{z+rj
\in \h \mid z \in \C, r \in \R^{+}\}$ of  the classical (real) quaternion algebra  $\h=\h(\frac{-1,-1}{\R})$. 
Denote by $\mbox{Iso}^{+} (\HQ^{i})$ the group of orientation-preserving isometries of $\ \HQ^{i}$ for $i=2,3$. 
It is well known that $\mbox{Iso}^{+}(\HQ^{2})$ is isomorphic with
$\PSL _2(\R)$ and  that $\mbox{Iso}^{+}(\HQ^{3})$ is isomorphic with
$\PSL _2(\C)$.

Throughout, we will use the notation $\begin{pmatrix} a & b \\ c & d \end{pmatrix}$ both for an element of $\SL_2(\R)$ or $\SL_2(\C)$ as well as for its natural image in $\PSL_2(\R)$ or $\PSL_2(\C)$. Moreover, abusing  notation, we  use the same letter for both the matrix in $\SL_2(\R)$ or $\SL_2(\C)$ and the M\"obius transformation acting on $\HQ^2$ or $\HQ^3$ respectively.
For $\gamma =\begin{pmatrix}
a & b\\
c & d
\end{pmatrix}$ in $\SL_2(\R)$ or $\SL_2(\C )$, we write $a=a(\gamma )$, $b=b(\gamma )$, $c=c(\gamma )$ and $d=d(\gamma )$ when it is necessary to stress the dependence of the entries on the matrix $\gamma$.

We now describe the action of $\PSL_2(\R)$ and $\PSL_2(\C)$ on hyperbolic space. We do this in detail for the $3$-dimensional case, the $2$-dimensional case being done in a similar way. The action of $\PSL_2(\C)$  on  $\HQ^{3}$  is given
by 

\begin{equation*}\begin{pmatrix}
a & b\\
c & d
\end{pmatrix} (P) = (aP+b)(cP+d)^{-1},
\end{equation*} 
where $(aP+b)(cP+d)^{-1}$  is calculated in $\h$. Explicitly, if $P=z+rj$ and $\gamma = \begin{pmatrix}
a & b\\
c & d
\end{pmatrix}$ then 
\begin{equation*}
\gamma (P)=\frac{(az+b)(\overline{c}\overline{z}+\overline{d})+a
\overline{c}r^{2}}{|cz+d|^{2}+|c|^{2}r^{2}} + (\frac{r}{|cz+d|^{2}+|c|^{2}r^{2}})j.
\end{equation*}

The hyperbolic distance $\rho$ in $\HQ^3$ (or $\HQ^2$ respectively)  is determined by  $\cosh \rho(P,P') = 1+\frac{d(P,P')^2}{2rr'},$ where $d$ is the Euclidean distance and $P=z+rj$ and $P'=z'+r'j$
are two elements of $\HQ^3$ (respectively $P=x+ri$ and $P'=x'+r'i$
are two elements of $\HQ^2$). 

Finally, recall that a group $\Gamma$ is said to act discontinuously on a proper metric space $X$ if for every compact subset $K$ of $X$, $K \cap \gamma(K) \neq \emptyset$ for only finitely many $\gamma \in \Gamma$. A well-known theorem states that if $X$ is a proper metric space, then a group $\Gamma$ acts discontinuously on $X$ if and only if $\Gamma$ is a discrete subgroup of $\textrm{Iso}(X)$. For more details on this, see \cite[Theorem 5.3.5]{ratcliffe}. A fundamental domain for a discontinuous group action of a group $\Gamma$ on a metric space $X$, is a closed set $\F \subseteq X$ such that the border of $\F$ has measure $0$, the union of the images of $\F$ under $\Gamma$ is the space $X$ and the images of $\F^{\circ}$, the interior of $\F$, by different elements of $\Gamma$ are pairwise disjoint. We call $\F$ a fundamental polyhedron, if $\F$ is a fundamental domain that is a convex polyhedron. The group $\Gamma$ is said to be cofinite if $\F$ has finite volume. It is well-known that cofiniteness implies geometrical finiteness which, in dimension $2$ and $3$, implies that the $\F$ has finitely many sides. For more details, we refer to \cite{bowditch} and \cite{kapovich}. If $\F$ is a fundemantal polyhedron for a discontinuous group action $\Gamma$, then, for every side $S$ of $\F$, there exists an element $\gamma_S \in \Gamma$ such that $S=\F \cap \gamma_S(\F)$. If for every side $S$ the element $\gamma_{S}$ is the reflection in the hyperplane $\langle S \rangle$, then $\Gamma$ is called a reflection group with respect to $F$. Reflection groups are particular cases of Coxeter 
groups. Formally a Coxeter group is a group $\langle r_1, \ldots, r_n \mid (r_ir_j)^{m_{ij}}=1 \rangle $ with $m_{ii}=1$ and $m_{ij} \geq 2$ for $i \neq j$. Note that $m_{ij}=\infty$ is possible and just means that there is no relation of the form $(r_ir_j)^m$ between $r_i$ and $r_j$. More details on this may be found in \cite[Section 7.1]{ratcliffe}.

In this paper we work with two different constructions of fundamental polyhedra, known as Dirichlet and Ford fundamental domain. We recall their construction and how they can be used to give a presentation for the considered groups, the so called Poincar\'e method (for details see for example \cite{beardon} or \cite{ratcliffe}). Let $\Gamma$ be a discrete subgroup of ${\rm Iso}^+(\HQ^3)$.    Let $\Gamma_j$ be the stabilizer in $\Gamma$ of $j\in \HQ^{3}$ 
and let
$\Fe_j$ be a convex fundamental polyhedron of $\Gamma_j$.
Put $D_{\gamma}(j)=\{u\in \HQ^{3}\mid \rho
(u,j)\leq \rho (u,\gamma (j))\}$ and set $\tilde{\F} = \bigcap_{\gamma\in
\Gamma \setminus \Gamma_j}D_{\gamma}(j)$.
The border $\partial D_{\gamma}(j) = \{u\in \HQ^3\mid \rho
(u,j)= \rho (u,\gamma (j))\}$ is the hyperbolic bisector of the geodesic linking $j$ to $\gamma(j)$. 
This is called a Poincar\'e bisector. Note that $\tilde{\F}$ is stable under the action of $\Gamma_j$. Moreover $\tilde{\F}$ is a convex polyhedron such that $\bigcup_{n} \gamma_n(\tilde{\F}) = \HQ^3$, where $\gamma_n$ are the coset representatives of $\Gamma_j$ in $\Gamma$. By \cite[Theorem 9.6.1]{beardon}, the set
$$\Fe=\F_j \cap \tilde{\F} = \Fe_j\cap
\left(\bigcap\limits_{\gamma\in
\Gamma \setminus \Gamma_j}D_{\gamma}(j)\right)$$ 
is a fundamental domain of $\Gamma$, which we call the Dirichlet fundamental domain with
center $j$. Moreover, it may be shown that $\F$ is a polyhedron and if  $\Gamma$ is   geometrically finite then 
 a finite set of generators for $\Gamma$ consists of the elements $\gamma \in \Gamma$ such that $\Fe \cap \gamma (\Fe )$ 
 is a side of the polyhedron  together with $\Gamma_j$, i.e. $\Gamma =\langle \Gamma_j, \gamma \  | \  \gamma (\Fe)
 \cap \Fe \textrm{ is a side }\rangle$ (see \cite[Theorem 6.8.3]{ratcliffe}). Let us denote the bisector $\partial D_{\gamma^{-1}}(j)$ of the geodesic linking $j$ to $\gamma^{-1}(j)$ by $\Sigma_{\gamma}$. It is easy to compute that $\gamma(\Sigma_{\gamma})=\Sigma_{\gamma^{-1}}$. From this it follows that $\F \cap \gamma(\F) \subseteq \Sigma_{\gamma^{-1}}$.
Note that the same construction can be done in $\HQ^2$, replacing the point $j \in \HQ^3$ by $i \in \HQ^2$.
 
Let $\Gamma$ be a discrete subgroup of $\PSL_2(\C)$ and denote by $\Gamma_{\infty}$ the stabilizer in $\Gamma$ of the point $\infty$. Denote a convex fundamental polyhedron of $\Gamma_{\infty}$ by $\F_{\infty}$. For $\gamma = \begin{pmatrix} a & b \\ c & d  
\end{pmatrix} \in \Gamma \setminus \Gamma_{\infty}$, denote the isometric sphere of $\gamma$ by  $\ISO_{\gamma}$.  Note that these are the points $P\in \HQ^3$ such that $\vert cP+d \vert^2 =1$. Denote the set
$\lbrace P\in \HQ^{3} \mid \vert cP+d\vert^2 \geq 1 \rbrace$ by $\ISO_{\gamma}^{\geq}$. By the same reasoning as above, if $\Gamma_{\infty}$ contains a parabolic element then
$$\F = \F_{\infty} \cap 
\left(\bigcap\limits_{\gamma\in\Gamma \setminus \Gamma_{\infty}} \ISO_{\gamma}^{\geq}\right)$$ 
is a fundamental domain of $\Gamma$ called the Ford fundamental domain of $\Gamma$. Again, if $\F$ is a polyhedron and if  $\Gamma$ is  geometrically finite then $\Gamma =\langle \Gamma_{\infty}, \gamma \  | \  \gamma (\Fe) \cap \Fe \textrm{ is a side }\rangle$. And also in this case one can easily show that, for every $\gamma \in \Gamma \setminus \Gamma_{\infty}$, $\gamma(\ISO_{\gamma})=\ISO_{\gamma^{-1}}$ and $\F \cap \gamma(\F) \subseteq \ISO_{\gamma^{-1}}$. Again the same construction is possible in $\HQ^2$.
\begin{remark}\label{non-cocompact}
If we talk about Ford domains, we implicitly assume the discrete subgroup $\Gamma$ to have a parabolic element fixing the point $\infty$.
\end{remark}



\section{Poincar\'e bisectors and isometric spheres
\label{seciso}}

For completeness' sake,  we first recall in this section the authors' result from \cite{algebrapaper} that is needed to prove our main result. Based on this result we prove a proposition that will be used later. As before, we develop the theory in dimension $3$, but everything can be applied to dimension $2$ as well.
Let $\gamma =
\left(
\begin{array}{ll}
a & b \\ c & d
\end{array}
\right)\in \PSL _2(\C ) \setminus  \PSU_2(\C)$.
Recall that $\gamma\in \PSU_2(\C )$ if and only if $\gamma (j)=j$ 
(see \cite{beardon, elstrodt}). 
As the Poincar\'e bisector can only exist if $\gamma \not \in \Gamma_j$, the case $\gamma \in \PSU_2(\C)$ is excluded in the following. 
 In the ball model of the hyperbolic space, it is well-known that the isometric sphere of $\gamma$ equals the bisector of the geodesic segment linking $0$ and its image by $\gamma^{-1}$. We will not go into  more details on this but refer the interested reader to \cite[Section 9.5]{beardon} for dimension $2$. In \cite[Theorem 3.1]{algebrapaper}, the authors give an independent proof of this in dimension $3$ (which is of course adaptable to dimension $2$).

In the upper half-space $\HQ^3$, an isometric sphere is not necessarily a Poincar\'e bisector. As explained in Section~\ref{secback}, we denote the isometric sphere of $\gamma$  by $\ISO_{\gamma}$ and the Poincar\'e bisector of the geodesic linking $j$ to $\gamma^{-1}(j)$ by $\Sigma_{\gamma}$.
This bisector may be a Euclidean sphere or a plane perpendicular to $\partial \HQ^3$.  If it is a Euclidean sphere,  we  denote its center by
$P_{\gamma}$  and its radius by $R_{\gamma}$.

The following result 
gives concrete formulas for the Poincar\'e bisectors in the upper half-space model. 

\begin{proposition}\cite[Proposition 3.2]{algebrapaper}
\label{isogammaupmodel}

Let $\gamma =
\left(
\begin{array}{ll}
a & b \\ c & d
\end{array}
\right) \in \PSL_2(\C )$, with $\gamma \not \in \PSU_2(\C)$. 
\begin{enumerate}
\item
$\Sigma_{\gamma} $ is a Euclidean sphere if and
only if $|a|^2+|c|^2 \neq 1$. In this case, its
center  and its radius are respectively given by
$P_{\gamma}=\frac{-(\overline{a}b+\overline{c}d)}{|a|^2+|c|^2-1}$  and 
$R^2_{\gamma}=\frac{1+\vert P_{\gamma}\vert^2}{|a|^2+|c|^2}$.
\item
$\Sigma_{\gamma}$ is a plane if and only if
$|a|^2+|c|^2 = 1$. In this case
$Re(\overline{v}z)+\frac{|v|^2}{2}=0, z\in \C$ is
a  defining equation of $\ \Sigma_{\gamma}$,  where
$v=\overline{a}b+\overline{c}d$.
\end{enumerate}
\end{proposition}

The next proposition gives some information on the relation between $\Iso_{\gamma}$ and $\Sigma_{\gamma}$, for some $\gamma \in \PSL_2(\C) \setminus \PSU_2(\C)$ with $c(\gamma) \neq 0$. Again the case $\gamma \in \PSU_2(\C)$ is excluded because otherwise $\Sigma_{\gamma}$ does not exist. This proposition will be useful in the study of DF domains and double Dirichlet domains.

\begin{proposition}
\label{isogammaupmodel2}

Let $\gamma =
\left(
\begin{array}{ll}
a & b \\ c & d
\end{array}
\right) \in \PSL_2(\C ) \setminus \PSU_2(\C)$. 
\begin{enumerate}
\item If $c\neq 0$ and $\vert a \vert^2 + \vert c \vert^2 \neq 1$ then
$\Iso_{\gamma}=\Sigma_{\gamma}$ if and only if
$d=\overline{a}$. In this
case, $\tr (\gamma )\in \R$ and if $b \neq 0$, $c=\lambda \overline{b}$,
with $\lambda \in \R$.  
\item If $c = 0$ and $\vert a \vert^2 = 1$, we also have that $a=\overline{d}$, $\tr (\gamma )\in \R$ and $c=\lambda \overline{b}$,
with $\lambda \in \R$. 
\end{enumerate}
\end{proposition}

\begin{proof} 
First note that in 1, the case $c=0$ is excluded so that $\ISO_{\gamma}$ exists and, by Proposition~\ref{isogammaupmodel}, the condition $\vert a \vert^2 + \vert c \vert^2 \neq 1$ guarantees that $\Sigma_{\gamma}$ is a sphere. If we denote the center of $\Iso_{\gamma}$ by $\hat{P}_{\gamma}$, then $|\hat{P}_{\gamma}-P_{\gamma}|=|-\frac{d}{c} +
\frac{\overline{a}b+\overline{c}d}{|a|^2+|c|^2-1}|=\frac{|d-\overline{a}|}{|c|(|a|^2+|c|^2-1)}$. Hence $\Iso_{\gamma}=\Sigma_{\gamma}$ implies that $d=\overline{a}$ and therefore $bc =|a|^2-1\in
\R$. The latter implies that $tr(\gamma )=a+\overline{a}\in \R$ and that $b=0$ or $c=\lambda\overline{b}$ for some $\lambda \in \R$. To prove the converse, suppose that $d=\overline{a}$. Then by the above $P_{\gamma}=\hat{P}_{\gamma}$. Moreover if $P_{\gamma}=-\frac{d}{c}$, then, by Proposition~\ref{isogammaupmodel}, $R_{\gamma}=\frac{1}{\vert c \vert}$ and hence $\Iso_{\gamma}=\Sigma_{\gamma}$. This proves the first item.

In the second item we have that the isometric sphere does not exist and $\Sigma_{\gamma}$ is a plane (and not a sphere). The conditions $c=0$  and $\vert a \vert^2  = 1$ imply that $ad=1$ and $a\overline{a}=1$. Hence 
$d=\overline{a}$ and $tr(\gamma) \in \R$. As $c=0$, the equality $c=\lambda \overline{b}$,
with $\lambda \in \R$ is trivially true.
\end{proof}

\begin{remark}\label{compareisobisc0}
Note that Proposition~\ref{isogammaupmodel2} does not treat the cases $c \neq 0$ and $\vert a \vert^2 + \vert c \vert^2 = 1$ and $c=0$ and $\vert a \vert^2 \neq 1$. In the first case the isometric sphere exists but $\Sigma_{\gamma}$ does not have the form of a sphere. In the second case $\Sigma_{\gamma}$ exist in the form of a sphere, but the isometric sphere does not exist. So in both cases it does not make sense to compare the isometric sphere with $\Sigma_{\gamma}$.
\end{remark}

\section{DF Domains and Double Dirichlet Domains}

The goal of this section is to prove Theorem~\ref{lakeland} and give some consequences that reprove and complement some  results in \cite{lakeland}.

The following definitions are taken from \cite{lakeland}. 

\begin{definition}\label{DFdef}
A Dirichlet  fundamental domain which is also a Ford domain in $\HQ^n$ is called a DF-domain.
%
A Dirichlet fundamental domain which has multiple centers is called a double Dirichlet Domain.
\end{definition}

Throughout this section, we work in $\HQ^2$ and $\HQ^3$ and we assume, without loss of generality, that the 
stabilizer of $i$, or $j$ respectively, in $\Gamma$, is trivial. The latter is possible by conjugating, if needed, 
the group $\Gamma$ by an adequate affine subgroup of $\PSL_2(\R)$, or $\PSL_2(\C)$ respectively. Indeed, denote by 
$\mathcal{A}$ the subgroup of $\PSL_2(\R)$, or $\PSL_2(\C)$ consisting of upper triangular matrices. Consider the conjugated group $\tau \Gamma \tau^{-1}$ of $\Gamma$ for some $\tau \in 
\mathcal{A}$ and let $P_0 \in \lbrace i, j \rbrace$, according to the space being $\HQ^2$ or $\HQ^3$ respectively. Then $
(\tau \Gamma \tau^{-1})_{P_0} = \tau \Gamma_{\tau^{-1}(P_0)} \tau^{-1}$ and thus if the stabilizer of $\tau^{-1}
(P_0)$ is trivial in $\Gamma$, the stabilizer of $P_0$ is trivial in $\tau \Gamma \tau^{-1}$. Let $\F$ be some 
fundamental domain for $\Gamma$. By definition every point in the interior of $\F$ has trivial stabilizer. As $
\mathcal{A}$ acts transitively on the upper half-plane, there exists $\tau \in \mathcal{A}$ such that $
\Gamma_{\tau^{-1}(P_0)}$ is trivial. Moreover, if $\Gamma$ contains a parabolic element of the form $1 \neq 
\begin{pmatrix} 1 & b \\ 0 & 1\end{pmatrix}$, then the conjugate $\tau \Gamma \tau^{-1}$ also contains such a 
parabolic element. So, instead of proving the results for $\Gamma$, we will prove them for a group conjugated to $
\Gamma$ with trivial stabilizer of $P_0$. It is easy to see that if $\tau \Gamma \tau^{-1}$ has a double Dirichlet 
domain, $\Gamma$ has a double Dirichlet domain. Similarly if $\tau \Gamma \tau^{-1}$ has a DF domain, $\Gamma$ has 
a DF domain.

We first give two lemmas on Fuchsian groups.

\begin{lemma}
\label{dfcriterium}
The following properties  are equivalent for $1 \neq \gamma \in \PSL _2(\R )$.
\begin{enumerate}
\item  $a(\gamma )=d(\gamma)$.
\item $\gamma =  \sigma\circ \sigma_{\gamma}$, where $\sigma$ denotes the reflection in the imaginary axes, i.e.,  $\sigma(z)=-{\overline{z}}$ and $\sigma_{\gamma}$ is the reflection in $\Sigma_{\gamma}$.
\item $\Sigma_{\gamma}$ is the bisector of the geodesic linking $ti$ and $\gamma^{-1}(ti)$, for all $t>0$ .
\item There exists $0<t_0\neq 1$ such that  $\Sigma_{\gamma}$ is the bisector of the geodesic segment linking $t_0i$ and $\gamma^{-1}(t_0i)$.
\end{enumerate}
 \end{lemma}
 
\begin{proof} 
We  first prove that 1 implies 2.
Suppose that $a(\gamma )=d(\gamma )$ and suppose first that $c(\gamma)=0$. Without loss of generality, we may take $a(\gamma )=d(\gamma )=1$. By Proposition~\ref{isogammaupmodel}, $\Sigma_{\gamma}$ is the line given by the equation $x=-\frac{b}{2}$ and thus $\sigma_{\gamma}(z) =\sigma (z+b(\gamma ))= \sigma (\gamma(z))$. If $c(\gamma) \neq 0$, by Proposition~\ref{isogammaupmodel2}, we  have that $\Sigma_{\gamma}=\Iso_{\gamma}$. Hence the  reflection $\sigma_{\gamma}$ in $\Sigma_{\gamma} $  is given by  $\sigma_{\gamma}(z)=P_{\gamma}-(|c|^2\sigma (z- P_{\gamma}))^{-1}=\sigma (\gamma (z))$. In either case, we have that  $\gamma =  \sigma\circ \sigma_{\gamma}$.  

Suppose now that $\gamma = \sigma\circ \sigma_{\gamma}$ and let  $u\in \Sigma_{\gamma}$. Then $\rho (u,\gamma^{-1}(ti))=\rho (u,\sigma_{\gamma}\circ\sigma (ti))$ $=\rho (u,\sigma_{\gamma}(ti))=\rho (\sigma_{\gamma}(u),ti)=\rho (u,ti)$ and hence $\Sigma_{\gamma}$ is the bisector of the geodesic linking $ti$ and $\gamma^{-1}(ti)$. This proves that 2 implies 3.  Obviously  3 implies 4.

We now prove that 4 item implies 1. Let $u\in \Sigma_{\gamma}$. Then we have that $\rho (u,t_0i)=\rho (u,\gamma^{-1}(t_0i))$ and hence $\rho (u,t_0i)=\rho (\gamma(u),t_0i)$. Since $\gamma$ is a M\"{o}bius transformation we have that ${\rm Im}(\gamma (z))=|\gamma^{\prime}(z)|{\rm Im}(z)$.  Using this and the explicit formula of the hyperbolic distance in the upper half-plane model (see Section \ref{secback}),  we obtain that $ |\gamma^{\prime}(u)||t_0i-u|^2=|t_0i-\gamma (u)|^2$.  It follows that ${\rm Re}(u)^2 |\gamma^{\prime}(u)| -{\rm Re }(\gamma (u))^2$ $=(1-|\gamma^{\prime}(u)|)t_0^2+(|\gamma^{\prime}(u)|-1)|\gamma^{\prime}(u)|{\rm Im} (u)^2$. We may write this as  an equation of the type $\alpha t^2=\beta$ having $t=t_0$ as a solution. However as $u \in \Sigma_{\gamma}$, by definition $\rho(u,i)=\rho(u,\gamma^{-1}(i))$ and hence also $t=1$ is also solution of the given equation. Thus we have that $\alpha=\beta$ and  $\alpha (t_0^2-1)=0$. It follows that $\alpha =0$ and thus $|\gamma^{\prime}(u)|=1$, for all $u\in \Sigma_{\gamma}$, i.e. $\Sigma_{\gamma}=\Iso_{\gamma}$.  Applying Proposition~\ref{isogammaupmodel2}, we obtain that $a(\gamma )=d(\gamma )$.
\end{proof}

Recall  that one says that  an angle $\alpha$ is a submultiple of an angle $\beta$  if either
there is a positive integer $n$ such that $n\alpha=\beta$ or $\alpha= 0$.

\begin{lemma}\label{corodfcriterium}
Let $\Gamma $ be a cofinite discrete subgroup of $\PSL_2(\R)$. Suppose that $i\in \HQ^2$  has  trivial stabilizer and let $\F$ be its Dirichlet fundamental polygon with center $i$.  Let $\gamma_k$ be the side-pairing transformations of $\F$ for $1 \leq k \leq n$.  If, for every $1 \leq k \leq n$, $a(\gamma_k)=d(\gamma_k)$, then $\Gamma$ is the subgroup of the orientation-preserving isometries of  a discrete reflection group.
\end{lemma}

\begin{proof}
First note that, as $a(\gamma_k)=d(\gamma_k)$, we have by Lemma~\ref{isogammaupmodel2} that $\Sigma_{\gamma_k}=\Iso_{\gamma_k}$. This means that $\Sigma_{\gamma_k^{-1}}$ has the same radius as $\Sigma_{\gamma_k}$ and their centers are the same in absolute value, but have opposite sign, and this for every $1 \leq k \leq n$. Hence $\F$ is symmetric with respect to the imaginary axis $\Sigma$. Consider the polygon $P$, whose sides are $\Sigma$ and the $\Sigma_{\gamma_k}$'s with $P_{\gamma_k}\geq 0$. We claim that all the dihedral angles of $P$
are submultiples of $\pi$. First we prove this statement for the dihedral angles between two sides of $\F$. For this, consider a vertex $E_k \subseteq \Sigma_{\gamma_k} \cap \Sigma_{\gamma_l}$ for 
$1 \leq k \neq l \leq n$. 
By the symmetry of $\F$, $\lbrace \Sigma_{\gamma_k}, \Sigma_{\gamma_{l}^{-1}} \rbrace$ is a finite sequence of sides of $\F$ determined by $\Sigma_{\gamma_k}$ and $E_k$, according to the definition of \cite[Chapter 6.8]{ratcliffe}.
By \cite[Theorem 6.8.7]{ratcliffe} the dihedral angle at the vertex $E_k$ is a 
submultiple of $\pi$ and this is true for every $1 \leq k \leq n$. Let $\Sigma_{\gamma_0}$ be the side of $\F$ with 
$P_{\gamma_0}\geq 0$ and such that $\Sigma_{\gamma_0} \cap \Sigma$ is a vertex of $P$ and consider the angle $\theta$ 
between $\Sigma$ and $\Sigma_{\gamma_0}$. If $P_{\gamma_0} > 0$, then the side $\Sigma_{\gamma_0^{-1}}$ has the 
same radius but $P_{\gamma_0^{-1}} < 0$. Hence, if  $\Sigma_{\gamma_0}$ intersects the imaginary axis $\Sigma$,
then so does $\Sigma_{\gamma_0^{-1}}$ and $\theta$ is half the the angle 
between $\Sigma_{\gamma_0}$ and $\Sigma_{\gamma_0^{-1}}$, which is a submultiple of $\pi$ by the previous. If 
$P_{\gamma_0}=0$, then $\Sigma$ is perpendicular to $\Sigma_{\gamma_0}$. Thus in both cases the angle between $
\Sigma$ and the adjacent side in $P$ is a submultiple of $\pi$. This proves the claim. Finally, by \cite[Theorem 
7.1.3]{ratcliffe}, the group $\tilde{\Gamma}=\langle  \sigma , \sigma_{\gamma_k} \mid P_{\gamma_k}\geq 0\rangle$, 
where $\sigma_k$ denotes the reflection in $\Sigma_{\gamma_k}$, is a discrete reflection group with respect to $P$. 
The result then follows by Lemma~\ref{dfcriterium}.
\end{proof}

We are now ready to prove Theorem~\ref{lakeland}.\\

\par\noindent{\bf Proof of Theorem~\ref{lakeland}.}  Let $\mathcal{F}$ be a DF domain, in $\HQ^2$ or $\HQ^3$ respectively, for $\Gamma$ with center $P_0\in \{i,j\}$. Let $\Phi_0$ be a set of  side-pairing transformations.
Consider $\gamma=\begin{pmatrix} a & b \\ c & d \end{pmatrix}\in \Phi_0$ with $c \neq 0$. As $\F$ is a Ford domain $\mathcal{F}\cap \gamma^{-1} (\mathcal{F}) \subseteq \Iso_{\gamma}$. As $\F$ is a also a Dirichlet domain $\mathcal{F}\cap \gamma^{-1} (\mathcal{F}) \subseteq\Sigma_{\gamma}$. Thus $\Iso_{\gamma}=\Sigma_{\gamma}$ and, as $\ISO_{\gamma}$ is a sphere, $\Sigma_{\gamma}$ is a sphere and hence $\vert a \vert^2 + \vert c \vert^2 \neq 1$. Thus, by Proposition~\ref{isogammaupmodel2}, $d=\overline{a}$. We now consider the case when $c=0$. Then, as $\F$ is a Ford domain, $\F \cap \gamma(\F)$ is a plane coming from the convex fundamental polyhedron of $\Gamma_{\infty}$. As $\F$ is a also a Dirichlet domain, $\Sigma_{\gamma}$ is a plane and hence $\vert a \vert^2 = \vert a \vert^2 + \vert c \vert^2 = 1$ and thus the second item of Proposition~\ref{isogammaupmodel2} allows to conclude. 

We now prove the converse. Let $\F$ be a Dirichlet or Ford fundamental domain and let $\Phi_0$ be a set of  side-pairing transformations, such that for every $\gamma=\begin{pmatrix} a & b \\ c & d \end{pmatrix} \in \Phi_0$, $d=\overline{a}$. 
Suppose first that $c \neq 0$, i.e. the isometric sphere associated to $\gamma$ exist. We claim that $\vert 
 a \vert^2 + \vert c \vert^2  \neq 1$. By contradiction, suppose the contrary. As $d=
\overline{a}$ and $det(\gamma)=1$, we have that $b=-\overline{c}$. Hence $\Vert \gamma 
\Vert^2=2\vert a \vert^2 + 2 \vert c\vert^2 =2$. By \cite[Theorem 2.5.1]{beardon}, $\gamma \in 
\PSU_2(\C)$ which is in contradiction with the fact that the stabilizer of $P_0$ is trivial. Hence by Proposition~
\ref{isogammaupmodel2}, $\Iso_{\gamma}=\Sigma_{\gamma}$. Suppose now that $c = 0$. Then $\F \cap \gamma(\F)$ is a plane coming from the convex fundamental polyhedron of $\Gamma_{\infty}$. The facts that 
$d=\overline{a}$ and $det(\gamma)=1$ imply that $\vert a \vert^2=1$ and thus $\Sigma_{\gamma}$ is a plane given by the equation $Re(a\overline{b}z)=-\frac{\vert b\vert^2}{2}$. By choosing the fundamental polyhedron of $\Gamma_{\infty}$ well, one of its sides coincides with $\Sigma_{\gamma}$. 

To prove the last part of Theorem~\ref{lakeland}, suppose that $\Gamma$ is Fuchsian.  By Lemma~\ref{corodfcriterium} we have that $\tilde{\Gamma}= \langle \sigma, \Gamma \rangle$  is a reflection group containing $\Gamma$ as a subgroup of index $2$. Consider finally the group $\hat{\Gamma}:=\langle \tau ,
 \Gamma \rangle$. 
It is clear that $\tau^ 2=1$ and by computation $(\tau \gamma)^2=1$ for all $\gamma \in \Phi_0$. Moreover, it is easy to compute that $(\tau \gamma)(\tau \gamma')$ has order at least $2$. It thus follows that
$\hat{\Gamma}$  is a Coxeter group with $[\hat{\Gamma} :\Gamma ]=2$.
$\qed$\par\bigskip.

Note that a presentation of $\ \tilde{\Gamma}$ and $\hat{\Gamma}$  can be obtained using \cite[Theorem II.7.5]{elstrodt}.  Also this result simplifies a lot the proof of \cite[Theorem 3.1]{lakeland}, i.e. it easily follows that the orbifold of $\Gamma$ is a punctured sphere in the Fuchsian case. Moreover, as is shown by the next corollary, \cite[Theorem 7.3]{lakeland} follows easily from Theorem~\ref{lakeland}.

\begin{corollary}\label{TraceRealC}
Let $\Gamma<\PSL_2(\C )$ be a cofinite discrete group and suppose $\Gamma$ admits a DF domain $\Fe$. Then, for every side-pairing transformation $\gamma$, $\tr (\gamma )\in \R$ and the vertical planes bisecting $\Sigma_{\gamma}$ and $\Sigma_{\gamma^{-1}}$ (for $\gamma \not \in \Gamma_{\infty}$) all intersect in a vertical axis.
\end{corollary}

\begin{proof}
Without loss of generality, we may assume that $\Gamma$ admits a DF domain with center $j$ (see the beginning of the section).
That $\tr (\gamma )\in \R$, for $\gamma$ a   side-pairing transformation,  is  a direct consequence of Theorem~\ref{lakeland} or of Proposition~\ref{isogammaupmodel2}.

If $\gamma \not \in \Gamma_{\infty}$, $\Sigma_{\gamma}$ and $\Sigma_{\gamma^{-1}}$ are Euclidean spheres with center $-\frac{\overline{a(\gamma)}}{c(\gamma)}$ and $\frac{a(\gamma)}{c(\gamma)}$ respectively. A simple computation then shows that the Euclidean bisector of these two points contains the point $0$ and hence the vertical plane bisecting $\Sigma_{\gamma}$ and $\Sigma_{\gamma^{-1}}$ contains the point $j$. Hence all these vertical planes intersect in a vertical line through $j$.
\end{proof}

We now consider when a fundamental domain is a double Dirichlet domain. The next two corollaries of our main Theorem give an alternative way to \cite[Section 4]{lakeland} to treat such domains.

\begin{corollary}\label{mainDFdouble}
Let $\Gamma$ be a cofinite Fuchsian group with trivial stabilizer of $i \in \HQ^2$. Then the following properties are equivalent.
\begin{enumerate}
\item $\Gamma$ is the subgroup of orientation-preserving isometries of a Fuchsian reflection group containing the reflection in the imaginary axis.
\item $\Gamma$ has a DF domain  with center $i$.
\item $\Gamma$ has a Dirichlet fundamental domain $\mathcal{F}$ with center $i$ such that, for every  side-pairing transformation $\gamma$, $a(\gamma)=d(\gamma)$. 
\item $\Gamma$ has a Dirichlet fundamental domain $\mathcal{F}$  with $i$ and $t_0i$ as centers, for some $1 \neq t_0 >0$.
\item $\Gamma$ has a Dirichlet fundamental domain $\mathcal{F}$ such that all the points of the geodesic through $i$ and $ti$, for $t >0$, are centers of $\mathcal{F}$.
\end{enumerate}
\end{corollary}

\begin{proof}  Theorem \ref{lakeland} shows that 2 and 3 are equivalent. The equivalence  of 3, 4 and 5  is given by Lemma \ref{dfcriterium}. Moreover, by Lemma~\ref{corodfcriterium}, 3  implies 1. We show that 1  implies 4. Fix a polygon $P$ for the reflection group such that one of the sides is the imaginary axis  $\Sigma$ and denote the reflection in $\Sigma$ by $\sigma$. Let $\Sigma_k$ be a side of $P$. Denote by $\sigma_{\gamma_k}$ the reflection in $\Sigma_k$ and let $\gamma_k=\sigma\circ \sigma_{\gamma_k}$. Then the result follows from Lemma~\ref{dfcriterium}.
\end{proof}

The previous result can be generalized, by conjugating the group $\Gamma$. We then get the following result, where the Dirichlet fundamental domain has arbitrary center $P \in \HQ^2$. However, in that case, the third item has to be dropped. We also regrouped parts 4 and 5. 

\begin{corollary}\label{lakelandcor}
Let $\Gamma$ be a cofinite Fuchsian group with trivial stabilizer of $P \in \HQ^2$.  The following properties are equivalent.
\begin{enumerate}
\item $\Gamma$ is the subgroup of orientation-preserving isometries of a Fuchsian reflection group containing the reflection in the vertical line through $P$.
\item $\Gamma$ has a DF domain with center $P$.
\item $\Gamma$ has a Dirichlet fundamental domain $\mathcal{F}$ such that all the points of the geodesic through $P$ and $P+i$ are centers of $\mathcal{F}$.
\end{enumerate}
\end{corollary} 

From Corollary~\ref{TraceRealC}, it follows that that all examples given in Section VII.3 in the book of Elstrodt, Grunewald and  Mennicke \cite{elstrodt} are
groups whose Ford domain is also a Dirichlet domain. Note that this does not follow immediately from the results of \cite{lakeland}.

Hence an interesting question  is to analyse when the Bianchi groups have a DF domain. This question can be linked to the following result by Belolipetsky and Mcleod \cite[Theorem 2.1]{belimcleod}: for the ring of integers $\O$ in $\Q (\sqrt{-d})$ (with $d$ a positive square free integer), the Bianchi group $\PSL_2(\O)$ extended by two reflections is a reflection group if and only if $d \leq 19$ and $d \neq 14,17$. 
To make a link to Belilopetsky's and Mcleod's result, we next show the following lemma.

\begin{lemma}\label{belilo2}
Let $d$ be a positive square free integer and let $\O$ be the ring of integers of 
$\Q(\sqrt{-d})$. Assume $d\neq 1,3$ and  let $\Gamma$ denote  the  Bianchi Group 
$\PSL_2(\O)$. Denote by $\sigma_x$ and $\sigma_y$ the reflections in the hyperplanes 
$x=0$ and $y=0$ in $\HQ^3$, i.e. for $P=z+rj \in \HQ^3$, $\sigma_x(P)= -\overline{P}$ and $\sigma_y(P) = \overline{z}+ rj$. Suppose that $\Gamma$ has a DF 
domain and, moreover, that here is a set of side-pairing transformations $\gamma$ for $\Gamma$ that 
have lower left matrix entry real or purely imaginary. Then $\langle \Gamma, \sigma_x, 
\sigma_y \rangle$ is a reflection group.
\end{lemma}

\begin{proof}
First note that the side-pairing transformations of a Dirichlet domain of center $j$ in the case of a Bianchi group are not uniquely determined, as the group has a non-trivial stabilizer of $j$, namely $\begin{pmatrix} 0 & 1\\ -1 & 0 \end{pmatrix}$. Nevertheless, if the Bianchi group $\Gamma$ has a DF domain, it is possible to choose the side-pairing transformations in such a way that the isometric sphere equals the bisector, i.e. $d=\overline{a}$ by Proposition~\ref{isogammaupmodel2}. Moreover, in this lemma we suppose the matrix entry $c$ of each side-pairing transformation to be real or purely imaginary.  

Let $\F$ be the Dirichlet domain of $\Gamma$ as described in \cite{algebrapaper}. Suppose moreover $\F$ is a DF domain and let $\gamma=\begin{pmatrix} a & b \\ c & \overline{a} \end{pmatrix}$ be a side-pairing transformation of $\F$ that does not fix $\infty$ with $c \in \R$ or $c \in i\R$. We first note that $\F$ is symmetric with respect to the hyperplane $x=0$. Indeed if  $\gamma \in \Gamma \setminus \Gamma_{\infty}$, then $\ISO_{\gamma}$ and $\ISO_{\gamma^{-1}}$ have the same radius and $\sigma_x(P_{\gamma})=P_{\gamma^{-1}}$. Moreover, if $\begin{pmatrix} a & b \\ c & \overline{a} \end{pmatrix} \in \Gamma \setminus \Gamma_{\infty}$ then also $\begin{pmatrix} \overline{a} & \overline{b} \\ \overline{c} & a \end{pmatrix} \in \Gamma \setminus \Gamma_{\infty}$. Thus we also have symmetry with respect to the hyperplane $y=0$. Let $\sigma_{\gamma}$ be the reflection in the isometric sphere $\ISO_{\gamma}$ of $\gamma$. Note that $\ISO_{\gamma}$ has center $-\frac{\overline{a}}{c}$ and radius $\frac{1}{\vert c \vert}$. We compute $\gamma\sigma_x(P)$ for $P= z+rj \in \HQ^3$. 
\begin{eqnarray*}
\gamma\sigma_{\gamma}(P) & = & \left(a \sigma_{\gamma}(P) + b\right)\left(c\sigma_{\gamma}(P) + \overline{a}\right)^{-1}\\
& = & c^{-1}\left(ac \sigma_{\gamma}(P) + bc\right)\left(c\sigma_{\gamma}(P) + \overline{a}\right)^{-1}\\
& = & c^{-1}\left(a(c\sigma_{\gamma}(P)+\overline{a})-(\vert a \vert^2-bc)\right)\left(c\sigma_{\gamma}(P) + \overline{a}\right)^{-1}\\
& = & \frac{a}{c}-c^{-1} \cdot 1 \cdot \left(c\left(-\frac{\overline{a}}{c} + \frac{P + \frac{\overline{a}}{c}}{\vert c \vert^2 \vert P + \frac{\overline{a}}{c} \vert^2}\right) + \overline{a} \right)^{-1}\\
& = & \frac{a}{c} - c^{-1}\left(\overline{c}\overline{P}+a\right)\\
& = & -c^{-1} \overline{P} \overline{c}\\
& = & -\frac{\overline{c}}{c} z + c^{-1}rj\overline{c}
\end{eqnarray*}
Suppose now that $c \in \R$ or $c \in i\R$. Then $\gamma\sigma_{\gamma}(P)=-\overline{P}$ or $\gamma\sigma_{\gamma}(P)=\overline{z}+rj$ and hence $\gamma=\sigma_x\sigma_{\gamma}$ or $\gamma=\sigma_y\sigma_{\gamma}$. We now have to distinguish two cases. Suppose first that $d \equiv 1,2 \mod 4$. Then, by \cite[Lemma 4.9]{algebrapaper}, the vertical hyperplanes of $\F$ are given by $x=\pm \frac{1}{2}$ and $y=\pm \frac{\sqrt{d}}{2}$. 
Define 
$$\tilde{\F}=\F \cap \lbrace z +rj \mid Re(z) \geq 0, Im(z) \geq 0 \rbrace.$$
Then the hyperplanes and spheres defining the border of $\tilde{\F}$ are the 
four hyperplanes given by $x=0$, $x=\frac{1}{2}$, $y=0$ and $y=\frac{\sqrt{d}}
{2}$, the spheres $\ISO_{\gamma}$ having center in $\lbrace z +rj \mid Re(z) 
\geq 0, Im(z) \geq 0 \rbrace$ and the unit sphere having center in $0$. This 
sphere comes from the unique stabilizer of $j$, given by the matrix 
$\begin{pmatrix} 0 & 1 \\ -1 & 0 \end{pmatrix}$. By the above, $\sigma_x\gamma$ or $\sigma_y\gamma$
is the reflection in $\ISO_{\gamma}$ if $\gamma$ does not fix $\infty$. If 
$\gamma$ fixes infinity then it is easy to see that $\sigma_x\gamma$ or 
$\sigma_y\gamma$ gives reflection in the vertcial hyperplanes. If 
$\gamma(j)=j$, then $\sigma_x\gamma=\sigma_y\gamma=\gamma$. As $\F$ is symmetric with respect 
to $x=0$ and $y=0$, $\tilde{\F}$ has finite volume. We now show that all 
dihedral angles are submutliples of $\pi$. Similar as in the proof of 
Lemma~\ref{corodfcriterium}, by the symmetry of $\F$ and \cite[Theorem 6.8.7]
{ratcliffe}, the dihedral angles of $\tilde{\F}$ are all submultiples of 
$\pi$. Hence, by \cite[Theorem 7.1.3]{ratcliffe}, $\langle \Gamma, \sigma_x, 
\sigma_y \rangle$ is a reflection group with respect to $\tilde{\F}$.

Suppose now that $d \equiv 3\  mod\  4$. Then, again by Lemma~\ref{corodfcriterium}, the vertical hyperplanes of $\F$ form a hexagon. This is not suited for a reflection polyhedron. Therefore we define $\tilde{\F}$ in the following way.
$$ \tilde{\F} = \left(\F \cup \theta(\F) \right) \cap \lbrace z +rj \mid 0 \leq Re(z) \leq \frac{1}{2}, 0 \leq Im(z) \leq \frac{(d+1)\sqrt{d}}{4d} \rbrace,$$
where $\theta$ is the translation given by $\begin{pmatrix} 1 & \frac{1+\sqrt{-d}}{2} \\ 0 & 1 \end{pmatrix}$. In this way, the vertical planes forming the border of $\tilde{\F}$ form a rectangle. Again, similar as above, the dihedral angles of $\tilde{\F}$ are all submultiples of $\pi$ and hence $\langle \Gamma, \sigma_x, 
\sigma_y \rangle$ is a reflection group with respect to $\tilde{\F}$.
\end{proof}

Using Aurel Page's package KleinianGroups~\cite{Aurel} and the algorithm described in \cite{algebrapaper}, we determine the side-pairing transformations for the Dirichlet fundamental domain with center $j$ for the Bianchi groups with $d \leq 19$. This yields the following result.

\begin{lemma}\label{BianchiDF}
Let $d$ and $\O$ be as in Lemma~ \ref{belilo2}. The Bianchi group $\PSL_2(\O)$ has a DF domain if and only if $d \in \lbrace 1,2,3,5,6,7,11,15,19 \rbrace$.
\end{lemma}

Combining Lemma~\ref{BianchiDF} and Lemma~\ref{belilo2}, we thus get a part of \cite[Theorem 2.1]{belimcleod}. Indeed for all the values of $d$ stated in Lemma~\ref{BianchiDF}, the condition that the matrix entry $c$ is real or purely imaginary is fulfilled and hence for these $d$ the Bianchi group is a subgroup of a reflection group. Note that we excluded $d=1$ and $d=3$ from Lemma~\ref{belilo2} because in that case the stabilizer of $j$ is more complicated. Nevertheless it can be easily verified, that also in these two cases, the Bianchi group $\PSL_2(\O)$ is a subgroup of index 4 of a reflection group This coincides in fact with a much earlier result of Bianchi. In \cite{bianchi}, Bianchi showed already that the Bianchi groups extended by two reflections are reflection groups for $d \leq 19$ and $d \neq 14,17$. For $d=10$ and $d=13$, Bianchi states that the group is generated by a so-called improper reflection, see \cite[Paragraph 17]{bianchi}. This is reflected in Lemma~\ref{BianchiDF}, which shows that the Bianchi group does not have a DF domain for these values of $d$.  \\

{\bf Acknowledgment.} The authors would like to thank the referee for pointing out the link between our results and \cite{belimcleod}. Moreover we thank Aurel Page for its help with the KleinianGroups package.

\bibliographystyle{abbrv}
\bibliography{KiJeJuSiSGEOBiblio}

\vspace{.25cm}

\noindent Department of Mathematics, \newline Vrije
Universiteit Brussel,\newline Pleinlaan 2, 1050
Brussel, Belgium\newline
emails: efjesper@vub.ac.be and akiefer@vub.ac.be

\vspace{.25cm}

\noindent Instituto de Matem\'atica e Estat\'\i
stica,\newline Universidade de S\~ao Paulo
(IME-USP),\newline Caixa Postal 66281, S\~ao
Paulo,\newline CEP  05315-970 - Brasil \newline
email: ostanley@usp.br

\vspace{.25cm}

\noindent Departamento de Matematica\\
Universidade Federal da Paraiba\\
e-mail: andrade@mat.ufpb.br

\vspace{.25cm}

\noindent Escola de Artes, Ci\^encias e
Humanidades,\newline Universidade de S\~ao Paulo
(EACH-USP),\newline Rua Arlindo B\'ettio, 1000,
Ermelindo Matarazzo, S\~ao Paulo, \newline CEP
03828-000 - Brasil \newline email:
acsouzafilho@usp.br
\end{document}